\documentclass[a4paper,11pt]{amsart}
\usepackage{amsmath}
\usepackage{cases}
\usepackage{amsfonts}
\usepackage[colorlinks,linkcolor=blue,citecolor=blue]{hyperref}
\usepackage{latexsym, amssymb, amsmath, amsthm, bbm}
\usepackage[all]{xy}
\usepackage{pgfplots}
\usepackage{mathtools}
\usepackage{mathrsfs}
\usepackage [latin1]{inputenc}
\usepackage[misc]{ifsym}
\usepackage{dsfont}
\usepackage{enumerate}
\usepackage{indentfirst}

\DeclareSymbolFont{EulerExtension}{U}{euex}{m}{n}
\DeclareMathSymbol{\euintop}{\mathop} {EulerExtension}{"52}
\DeclareMathSymbol{\euointop}{\mathop} {EulerExtension}{"48}

\allowdisplaybreaks[4]

\def \Ob{\operatorname{Ob}}
\def \id{\operatorname{Id}}
\def \ker{\operatorname{Ker}}

\def \im{\operatorname{im}}

\def \dim{\operatorname{dim}}

\def \Id{\operatorname{Id}}

\def \Rep{\operatorname{Rep}}

\def \id{\operatorname{Id}}
\def \ker{\operatorname{Ker}}

\def \im{\operatorname{im}}

\numberwithin{equation}{section}

\newtheorem{theorem}{Theorem}[section]
\newtheorem{main}[theorem]{Main Theorem}
\newtheorem{lemma}[theorem]{Lemma}
\newtheorem{proposition}[theorem]{Proposition}
\newtheorem{corollary}[theorem]{Corollary}
\newtheorem{definition}[theorem]{Definition}
\newtheorem{example}[theorem]{Example}
\newtheorem{remark}[theorem]{Remark}

\newtheorem{convention}[theorem]{Convention}
\newtheorem{conjecture}[theorem]{Conjecture}

\begin{document}

\title{Splitting Property of quasitriangular Hopf algebras}
\thanks{$^*$Corresponding author}
\author{Jinsong Wu} \address{ Beijing Institute of Mathematical Sciences and Applications, Beijing 101408, China. }\email{wjs@bimsa.cn}
\author{Kun Zhou*}\address{Beijing Institute of Mathematical Sciences and Applications, Beijing 101408, China. }\email{kzhou@bimsa.cn}


\date{}
\maketitle
\begin{abstract}
We investigate the splitting property of quasitriangular Hopf algebras through the lens of twisted tensor products.
Specifically, we demonstrate that an infinite-dimensional quasitriangular Hopf algebra possesses the splitting property if it admits a factorizable quotient Hopf algebra.
We also establish that the splitting property holds if there exists a full rank quotient Hopf algebra where the left and right coinvariants coincide.
As a consequence, we obtain a new obstruction for quasitriangular structure in finite-dimensional Hopf algebras.
\end{abstract}


\section{Introduction}

Quasitriangular Hopf algebras, introduced by Drinfeld \cite{VG}, provide solutions to the quantum Yang-Baxter equations, a cornerstone in the study of quantum integrable systems.
The small quantum group $u_q(\mathfrak{g})$, with appropriate parameter $q$, is an example of a quasitriangular Hopf algebra \cite{DA, KA1, Ro, LG-M, LG-Q, SSQ}.
Representations of finite-dimensional quasitriangular Hopf algebras give rise to braided finite tensor categories, which serve as a rich and significant source of such structures. Braided finite tensor categories play a central role in various fields of mathematics and theoretical physics due to their profound connections with quantum groups, knot theory, and topological quantum field theory \cite{Ro, Fre, Ker, TV1, NV}.

A natural question arises: can a quasitriangular Hopf algebra be expressed as a tensor product of two non-trivial quasitriangular Hopf algebras? The categorical analogue of this question asks whether a braided finite tensor category is braided equivalent to the Deligne tensor product of two braided finite tensor categories.
It is known that a semisimple braided finite tensor category is braided equivalent to the Deligne tensor product of a non-degenerate braided tensor subcategory and its relative M\"{u}ger centralizer (See \cite{EP} for details).
In 2019, K. Shimizu \cite{Shi} established dimension formulas for relative M\"{u}ger centralizers in the non-semisimple case.
Building on these results, R. Laugwitz and C. Walton \cite{LC} extended the conclusion to non-semisimple settings in 2022.
Moreover, two representation categories of quasitriangular Hopf algebras are braided equivalent if and only if one quasitriangular Hopf algebra is Hopf isomorphic to a twist of the other.
As a result, any finite-dimensional quasitriangular Hopf algebra with a factorizable quotient Hopf algebra can be expressed as a twisted tensor product of two quasitriangular Hopf algebras.
This leads us to modify the question: can a quasitriangular Hopf algebra be expressed as a twisted tensor product of two non-trivial quasitriangular Hopf algebras?
We say a quasitriangular Hopf algebra has splitting property if it is a twisted tensor product of two non-trivial quasitriangular Hopf algebras.

Their methods rely heavily on dimension formulas.
In contrast, we propose an alternative approach that allows us to eliminate the finiteness assumption.
In this paper, we establish a series of sufficient conditions for the splitting property of quasitriangular Hopf algebras.

\begin{main}[Theorems \ref{thm-m1}, \ref{thm-m2}]
Suppose that $(H, R)$ is a quasitriangular Hopf algebra ( which may be infinite dimensional) and $(K,(\pi\otimes \pi)(R))$ is a nontrivial quotient Hopf algebra of $(H, R)$.
Suppose that one of the following conditions holds
\begin{itemize}
\item $(K,(\pi\otimes \pi)(R))$ is factorizable.
\item $H^{co\pi}=\prescript{co\pi}{}{H}$ and $(K,(\pi\otimes \pi)(R))$ is full rank.
\end{itemize}
Then $(H, R)$ has splitting property.
\end{main}

When the quasitriangular Hopf algebra $H$ is an extension, we have the following results.

\begin{main}[Theorem \ref{pro-ext}]
Suppose that $(H, R)$ is a quasitriangular Hopf algebra (may infinite dimension) and one of the following conditions holds:
\begin{itemize}
\item $H$ is an extension of a Hopf algebra $A$ by $K$ and $K$ is factorizable.
\item $H$ is an extension of a Hopf algebra $A$ by $K$ and $K$ is full rank.
\end{itemize}
  Then $A$ admits a quasitriangular structure and $H$ is a twisted tensor product of $K$ and $A$.
\end{main}

In general, identifying obstructions to the existence of quasitriangular structures in finite-dimensional Hopf algebras is challenging. However, leveraging the splitting property, we derive a concrete obstruction for such structures.
\begin{main}[Theorem \ref{pro-qtt}]
Suppose $H$ is a finite dimensional Hopf algebra.
Suppose $|G(H^*)|\neq 1$, $G(H^*)\cap Z(H^*)=\{1\}$, $G(H)\cap Z(H)=\{1\}$, and there exists an odd prime factor $p$ of $|G(H^*)|$ such that for any $g\in G(H)$ with order $p$ and any $\alpha \in G(H^*)$ with order $p$, $\alpha(g)\neq 1$.
Then $H$ don't have any universal $R$-matrix.
\end{main}
Notably, this criterion depends solely on the group parts of Hopf algebras.
We anticipate that this approach will provide a convenient method for excluding certain Hopf algebras without quasitriangular structures.

The existing criteria for determining whether a Hopf algebra $H$ is quasitriangular are very few, such as checking if the representation ring of $H$, is commutative or if the antipode of $H$ is inner ($S^2(h)=uhgu^{-1}$ for some invertible element $u\in H$ or $S^4(h)=ghg^{-1}$ for some $g\in G(H)$).
However, these criteria are often not sufficient.
For example the Taft algebra $H_{p^2}$ (Example \ref{pro-exsit1}) with dimension $p^2$ for an odd prime $p$ satisfies both the commutativity of its representation ring (\cite{CHFO}) and the inner property of its antipode, but it admits no quasitriangular structure, which can be deduced by our obstruction.

The remarkable theorem of S. Natale (\cite[Theorem 1.2]{Nar}) plays a pivotal role in addressing a question posed by N. Andruskiewitsch \cite[Question 6.5]{AN-que} in the context of quasitriangular Hopf algebras.
The original proof of this theorem is intricate and profound.
Here, we aim to provide a shorter proof by leveraging our results.

This paper is organized as follows.
Section 2, we review the Hopf algebra in a braided monoidal category and the Cleft extension.
Section 3, we study the twisted tensor product and revisit the splitting property of finite dimensional quasitriangular Hopf algebras.
Section 4, we study the splitting property for (possibly infinite dimensional) quasitriangular Hopf algebras and obtain a series sufficient conditions for the splitting property.
Section 5, we study the splitting property for extensions of Hopf algebras.
Section 6, we obtain an obstruction of quasitriangular structure in finite dimensional Hopf algebras.
Moreover, we present a short proof of Natale's result in \cite{Nar}.

\begin{convention}\emph{Throughout this paper, we work over arbitrary field $\Bbbk$. For the symbol $\delta$, we mean the classical Kronecker's symbol. Our references for the theory of Hopf algebras are \cite{R,MS}, and for the theory of tensor categories, we refer to \cite{EP}. }
\end{convention}

\section{Preliminaries}

\subsection{Quasitriangular Hopf Algebras}

A \emph{quasitriangular Hopf algebra} is a pair $(H, R)$, where $H\equiv(H, m, \eta, \Delta, \epsilon, S)$ is a Hopf algebra over $\Bbbk$ and $\displaystyle R=\sum_i R_i \otimes R^i$ is the $R$-matrix.
Precisely, $R$ is an invertible element in $H\otimes H$ such that
\begin{equation*}
 (\Delta \otimes \id)(R)=R_{13}R_{23},\; (\id \otimes \Delta)(R)=R_{13}R_{12},\;,
 \end{equation*}
and $\Delta^{op}(h)=R\Delta(h)R^{-1}$ for all $h\in H$, where $\Delta^{op}=\tau\circ \Delta$, and $\tau$ is the flip map.
Here, $\displaystyle R_{12}= \sum_i R_i \otimes R^i\otimes 1 $, with similar definitions for $R_{13}$ and $R_{23}$.

Given a Hopf algebra $(H,m,\eta,\Delta,\epsilon,S)$, we write $H^{cop}$ (resp. $H^{op}$) as the Hopf algebra $(H,m,\eta,\tau\circ\Delta,\epsilon,S)$ (resp. $(H,m\circ \tau,\eta,\Delta,\epsilon,S)$). And we will use Sweedler's notation for its coproduct.

The finite dual $H^o$ of $H$ is defined by
\begin{align*}
H^o=\{f\in H^* |\;f(\mathcal{I}_f)=0 \text{ for some finite codimensional ideal } \mathcal{I}_f  \text{ in } H\}.
\end{align*}
The Hopf algebra structure on $H^o$ is induced by $H$.(see \cite{R} for example).

Suppose $(H_1,R_1),(H_2,R_2)$ are quasitriangular Hopf algebras.
We say $(H_1, R_1)$ and $(H_2, R_2)$ are isomorphic as quasitriangular Hopf algebras if there exists a Hopf isomorphism $\Psi:H_1\rightarrow H_2$ that preserves their quasitriangular structures, meaning $(\Psi\otimes \Psi)(R_1)=R_2$.

A quasitriangular Hopf algebra $(H,R)$ is called \emph{a ribbon Hopf algebra} if there exists a ribbon element in $H$, i.e. a central element $\theta$ in $H$ satisfying the relations: $\epsilon(\theta)=1$, $S(\theta)=\theta$ and
$$ \Delta(\theta)=(R_{21}\cdot R)^{-1}(\theta\otimes \theta),$$
where
\begin{align*}
    R_{21}\cdot R=(m\otimes m) (\id\otimes \tau\otimes \id)((\tau R)\otimes R).
\end{align*}
Two ribbon Hopf algebras are said to be isomorphic as ribbon Hopf algebras if there exists a Hopf isomorphism between them that preserves both their quasitriangular structures and their ribbon elements.

We define the linear maps $\Phi_R: H^{o}\rightarrow H$ and $\Phi_{\tau R}: H^{o}\rightarrow H$ as follows:
\begin{equation}\label{eq2.5}
\begin{aligned}
\Phi_R(f):&  =(f\otimes \Id)(R_{21}\cdot R), \;f\in H^o.\\
\Phi_{\tau R}(f): & =(\Id\otimes f)(R_{21}\cdot R),\;f\in H^o.
\end{aligned}
\end{equation}
A \emph{factorizable Hopf algebra} is a finite dimensional quasitriangular Hopf algebra $(H, R)$ for which $\Phi_R$ or equivalently $\Phi_{\tau R}$, is a linear isomorphism.
Suppose $C\subseteq H^o$ be a subcoalgebra.
Then $\Phi_R(C)\subseteq H$ is a normal left coideal of $H$ (See \cite[Lemma 4.1]{Nar}.)

\begin{remark}\label{rk-nomco}
It's known that $\Phi_R(H^o)$ is a normal left coideal subalgebra of $H$ (see \cite[Chapter 9]{MSaF}).
More generally, if $B\subseteq H^o$ is a Hopf subalgebra, then $\Phi_R(B)$ is a normal left coideal subalgebra of $H$ (see \cite[Section 3]{Nar}).
\end{remark}

Suppose $(K, \pi)$ is a quotient Hopf algebra of $(H, R)$, where $\pi:H\to K$ is a surjective Hopf map.
Let
\begin{align*}
    H^{co\pi}=& \{h\in H|\;(\id\otimes \pi)\circ \Delta=h\otimes 1\},\\
    \prescript{co\pi}{}{H}=&\{h\in H|\;(\pi\otimes \id)\circ \Delta=1\otimes h\}.
\end{align*}
$H^{co\pi}$ (resp.  $\prescript{co\pi}{}{H}$) is normal left (resp. right) coideal subalgebra of $H$.


For the quasitriangular Hopf algebra $(H,R)$, there are Hopf maps $l_R: (H^o)^{cop}\rightarrow H$ and $r_R: (H^o)^{op}\rightarrow H$, given respectively by
$$l_R(f):=(f\otimes \Id)(R),\;\;r_R(f):=(\Id \otimes f)(R),\;f\in H^o.$$
If the above map $l_R$ is surjective, we say $(H,R)$ is \emph{full rank}.

\subsection{Hopf Algebras in Braided Monoidal Category}

We shall briefly recall the Hopf algebras in a braided monoidal category in this section.
Suppose $(\mathcal{D},\otimes,\mathds{1},c)$ is a strict braided monoidal category, where $c$ is a braiding structure on $(\mathcal{D},\otimes,\mathds{1})$.

For an algebra $(A,m_A,\eta_A)$ and a coalgebra $(B,\Delta_B,\epsilon_B)$ in $\mathcal{D}$, we denote by $\text{Reg}(B,A)$ the group of invertible elements in $\mathcal{D}(B,A)$ (the set of morphisms in $\mathcal{D}$ from $B$ to $A$), where the convolution product given by
\begin{align*}
f\ast g = m_A\otimes (f\otimes g)\circ \Delta_B, \quad f, g\in \mathcal{D}(B, A).
\end{align*}
We see that the unit element of $\text{Reg}(B,A)$ is $\eta_A \otimes \epsilon_B$.
Now we recall the Hopf algebras in braided monoidal category (see \cite{AAj, MSa, TMs, TMf} for details).

\begin{definition}[Hopf Algebra in Category]
A Hopf algebra in $(\mathcal{D},\otimes,\mathds{1},c)$ is a six-tuple $(H,m,\eta,\Delta,\epsilon,S)$ such that
\begin{itemize}
  \item[(i)] $(H,m,\eta)$ is an algebra in $\mathcal{D}$ and $(H,\Delta,\epsilon)$ is a coalgebra in $\mathcal{D}$,
  \item[(ii)]  $\Delta,\epsilon\;$ are algebra maps in $\mathcal{D}$,
  \item[(iii)] $S$ is convolution inverse to the identity map in \emph{Reg}$(H,H)$.
\end{itemize}
\end{definition}

We recall the Hopf algebra in the representation category of a quasitriangular Hopf algebras, i.e. Majid's transmutation theory \cite[Chapter 9.4]{MSaF}.
Let $(H, R)$ be a quasitriangular Hopf algebra.
Majid shows that there is a Hopf algebra $\underline{H}$ in the braided monoidal category $\Rep(H)$ corresponding to $(H,R)$, here $\Rep(H)$ means that the category of representation of $H$ (contain infinite dimensional representation) where the braiding structure is $\tau \circ R$, which we shall describe in the following.
The algebra structure of the Hopf algebra $\underline{H}$ is the same as $H$.
We shall identify the elements in $\underline{H}$ with the elements in $H$.
We shall describe the coalgebra structure of $\underline{H}$.
The action of $H$ on $\underline{H}$ is defined by
\begin{align*}
\text{ad}_h(a):= h_{(1)}aS(h_{(2)}), \quad h\in H, \ a\in \underline{H},
\end{align*}
the comultiplication $\underline{\Delta}$ in $\Rep(H)$ is given by:
\begin{equation}
\underline{\Delta}(a)= \sum a_{(1)} \otimes a_{(2)}:=\sum_{i}a_{(1)}S(R^i)\otimes \text{ad}_{R_i}( a_{(2)}), \label{copro}
\end{equation}
and the antipode $\underline{S}$ is given by
\begin{equation}
\underline{S}(a):= \sum_{i}R^iS(\text{ad}_{R_i}(a)). \label{anti}
\end{equation}

There is a Hopf algebra structure on the dual space $\underline{H^o}$ in the braided tensor category $\Rep(H)$.
We describe as follows.
The coalgebra structure of the Hopf algebra $\underline{H^o}$ is the same as $H^o$.
The action of $H$ on $\underline{H^o}$ defined by
\begin{align*}
\langle h \triangleright f,a\rangle=\langle f,S(h_{(1)})ah_{(2)}\rangle, \quad h\in H, \ a\in \underline{H} ,\  f\in \underline{H^o}
\end{align*}
with multiplication $\underline{m}$ is given by:
\begin{equation}
f\text{\underline{.}}g:=\langle S(f_{(1)})f_{(3)},S(g_{(1)})\rangle_R\  f_{(2)}g_{(2)} , \label{pro}
\end{equation}
the antipode $\underline{S}$ is given by
\begin{equation}
\underline{S}(f):=\langle S^2(f_{(3)})S(f_{(1)}), f_{(4)}\rangle_R \ S(f_{(2)}), \label{anti-dual}
\end{equation}
where $\langle f,g\rangle_R:=(f\otimes g)(R)$ for all $f,g\in H^o$ and
\begin{align*}
    \Delta^{(3)} (f)=f_{(1)}\otimes f_{(2)} \otimes f_{(3)} \otimes f_{(4)}.
\end{align*}

The map $\Phi_R: H^o\to H$ induces a map $\underline{\Phi_R}:\underline{H^o} \rightarrow \underline{H}$ by $\underline{\Phi_R}(f)=\Phi_R(f)$, $f\in \underline{H^o}$.
Majid's Transmution theory \cite[Chapter 9]{MSaF} tells us that $\underline{\Phi_R}$ is a Hopf map in $\Rep(H)$.

Suppose $(K, \pi)$ is a quotient Hopf algebra of $(H, R)$.
We denote $\overline{R}=(\pi\otimes \pi)(R)$.
The algebra structure of the Hopf algebra $\underline{K}$ is the same as $K$ and the co-algebra structure is described in the following.
The comultiplication $\underline{\Delta}$ is given by
\begin{equation}\label{eq-cok}
\underline{\Delta}(k)= \sum k_{(1)} \otimes k_{(2)}:=\sum_{i}k_{(1)}\pi(S(R^i))\otimes \text{ad}_{\pi(R_i)}( k_{(2)}), \quad k \in \underline{K},
\end{equation}
the antipode $\underline{S}$ is given by
\begin{equation}
\underline{S}(k):= \sum_{i}\pi(R^i)S(\text{ad}_{\pi(R_i)}(k)), \quad k\in \underline{K}.
\end{equation}
The action of $H$ on $\underline{K}$ is given by
\begin{align*}
h.k:=\text{ad}_{\pi(h)}(k), \quad h\in H,\ k\in \underline{K}.
\end{align*}
Now we see that $\underline{K}$ is a Hopf algebra in $\Rep(H)$.

Similarly, the finite dual $\underline{K^o}$ is a Hopf algebra in $\Rep(H)$.
The coalgebra structure of the Hopf algebra $\underline{K^o}$ is the same as $K^o$.
The action of $H$ on $\underline{K^o}$ is given by
\begin{align*}
\langle h \triangleright f,k\rangle=\langle f,S(\pi(h_{(1)}))k\pi(h_{(2)})\rangle, \quad h\in H, k\in \underline{K}, f\in \underline{K^o}.
\end{align*}

Moreover, the multiplication and the antipode of $\underline{K^*}$ in $\Rep(H)$ are given by:
\begin{equation}
f\text{\underline{.}}g:=\langle S(f_{(1)})f_{(3)},S(g_{(1)})\rangle_{\overline{R}} (f_{(2)}g_{(2)}),
\end{equation}
\begin{equation}
\underline{S}(f):=\langle S^2(f_{(3)})S(f_{(1)}), f_{(4)}\rangle_{\overline{R}} S(f_{(2)}),
\end{equation}
where $\langle f,g\rangle_{\overline{R}}:=(f\otimes g)(\overline{R})$ for all $f,g\in K^*$.

The linear map $\underline{\pi}: \underline{H} \to \underline{K}$ is defined as $\underline{\pi}(a)=\pi(a)$, $a\in \underline{H}$.
We shall show that $\underline{\pi}$ is a Hopf map preserving the coalgebra structures of $\underline{H}$ and $\underline{K}$.

Suppose $B$ is a Hopf algebra in a braided monoidal category $\mathcal{D}$.
Recall that an algebra $(A,m_A,\eta_A)$ is a right $B$-comodule algebra if:
\begin{itemize}
  \item[(i)] There is a morphism $\rho_A\in \mathcal{D}(A, A\otimes B)$ such that
      $$(\rho_A \otimes \Id_B)\circ \rho_A=(\Id_A \otimes \Delta_B)\circ \rho_A,\;(\Id_A \otimes \epsilon_B)\circ \rho_A=\Id_A,$$
    i.e.  $(A,\rho_A)$ is a right $B$-comodule.

  \item[(ii)] $\eta_A$ and $m_A$ are morphisms of right $B$-comodules with respect to $\rho_A$, i.e. the following equations hold
      \begin{align*}
      \rho_A \circ \eta_A=&\eta_A\otimes \eta_B, \\
      \rho_A\circ m_A=&(m_A\otimes \Id_B)\circ \rho_{A\otimes A},
      \end{align*}
\end{itemize}
where $\rho_{A\otimes A}:=(\Id_{A\otimes A}\otimes m_B)(\Id_A\otimes c_{B,A}\otimes \Id_A)\circ(\rho_A\otimes \rho_A)$.

\begin{proposition}\label{prop-hopf}
Suppose that $(H, R)$ is a quasitriangular Hopf algebra and $(K,\pi)$ a quotient Hopf algebra of $(H, R)$.
We have that $\underline{\pi}$ is a Hopf map in $\Rep(H)$ and  $(\underline{H}, (\id\otimes \underline{\pi})\circ \underline{\Delta})$ is a right $\underline{K}$-comodule algebra in $\Rep(H)$.
\end{proposition}

\begin{proof}
Note that $\underline{H}, \underline{K}\in \Rep(H)$ and $\pi: H\to K$ is a Hopf map.
We see that $\underline{\pi}$ is a morphism in $\Rep(H)$ and $\underline{\pi}$ is an algebra map in $\Rep(H)$.
Now we will show that $\underline{\pi}$ is a coalgebra map in $\Rep(H)$.
For any $h\in \underline{H}$, we have that
\begin{align*}
(\underline{\pi}\otimes \underline{\pi})\circ \underline{\Delta}(h)&= (\underline{\pi}\otimes \underline{\pi})(\sum_{i}h_{(1)}S(R^i)\otimes \text{ad}_{R_i}( h_{(2)}))\\
&=\sum_{i}\pi(h_{(1)})S(\pi(R^i))\otimes \pi[(R_i)_{(1)})h_{(2)}S((R_i)_{(2)})]\\
&=\sum_{i}\pi(h_{(1)})S(\pi(R^i))\otimes \pi((R_i)_{(1)})\pi(h_{(2)})\pi(S((R_i)_{(2)}))\\
&=\sum_{i}\pi(h_{(1)})S(\pi(R^i))\otimes \pi(R_i)_{(1)})\pi(h_{(2)})S(\pi(R_i)_{(2)})\\
&=\sum_{i}\pi(h)_{(1)} S(\pi(R^i))\otimes \pi(R_i)_{(1)})\pi(h)_{(2)} S(\pi(R_i)_{(2)})\\
&=\sum_{i}\pi(h)_{(1)} S(\pi(R^i))\otimes \text{ad}_{\pi(R_i)}(\pi(h)_{(2)})\\
&=\underline{\Delta}\circ \underline{\pi} (h).
\end{align*}
This implies that $\underline{\pi}$ is a coalgebra map in $\Rep(H)$.
To show that $(\underline{H},(\id\otimes \underline{\pi})\circ \underline{\Delta})$ is a right $\underline{K}$-comodule algebra in $\Rep(H)$, the only non-trivial case is to prove that $m_{\underline{H}}$ is a comodule map with respect to the comodule structure $(\id\otimes \underline{\pi})\circ \underline{\Delta}$, i.e. the following equality holds:
\begin{equation}\label{equ-com}
\begin{aligned}
 & (\id_{\underline{H}}\otimes \underline{\pi})\circ \Delta_{\underline{H}} \circ m_{\underline{H}}\\
 & = (m_{\underline{H}} \otimes m_{\underline{K}} ) \circ (\id_{\underline{H}} \otimes c_{\underline{K},\underline{H}} \otimes
\id_{\underline{H}})  \circ (\id_{\underline{H}} \otimes \underline{\pi} \otimes
\id_{\underline{H}}\otimes \underline{\pi})\circ (  \Delta_{\underline{H}} \otimes
 \Delta_{\underline{H}}).
 \end{aligned}
\end{equation}
Using the naturality of the braiding stucture, we get $c_{\underline{K}\otimes \underline{H}}\circ (\pi\otimes \id_{\underline{H}})= (\id_{\underline{H}}\otimes \pi )\circ c_{\underline{H}\otimes \underline{H}}$.
Combining this with the fact that $\Delta_{\underline{H}}$ is algebra map in $\Rep(H)$, we obtain that the above equality \eqref{equ-com} holds.
\end{proof}
Let $\underline{H^{co\pi}}=\{h\in \underline{H}|\;(\id\otimes \underline{\pi})\circ \underline{\Delta}(h)=h\otimes 1\}$.
\begin{lemma}\label{lem-copp}
We have $\underline{H^{co\pi}}=H^{co\pi}$ as vector spaces.
\end{lemma}

\begin{proof}
Let $h\in \underline{H^{co\pi}}$.
We have $\displaystyle \sum_i h_{(1)}S(R^i)\otimes \underline{\pi}(\text{ad}_{R_i}(h_{(2)}))=h\otimes 1$. Thus, we obtain
$$\sum_{i,j} h_{(1)}S(R^i)S(R^j)\otimes \text{ad}_{\pi\circ S(R_j)}[\underline{\pi}(\text{ad}_{R_i}(h_{(2)}))]=h\otimes 1,$$
Combing this equation with $R^{-1}=(S\otimes \Id)(R)$, we know $(\Id\otimes \pi)\circ \Delta (h)=h\otimes 1$, i.e. $h\in H^{co\pi}$. Thus, $\underline{H^{co\pi}}\subseteq H^{co\pi}$. Similarly, we get $H^{co\pi}\subseteq \underline{H^{co\pi}}$. Hence, we complete the proof.
\end{proof}


\subsection{Cleft Extension}
We recall cleft extensions in a braided monoidal category $\mathcal{D}$ (see \cite{AAj} for details) in this section.
Suppose $H$ is a Hopf algebra in $\mathcal{D}$.

Let $A$ be an algebra in $\mathcal{D}$ and $\mu_A\in \mathcal{D}(H\otimes A,A)$.
The morphism $\mu_A$ is an action of $H$ on $A$ if it satisfies:
\begin{itemize}
  \item[(i)] $\mu_A\circ (\eta_H\otimes \Id_A)=\Id_A$;
  \item[(ii)] $\mu_A\circ (\Id_H\otimes \eta_A)=\epsilon_H\otimes \eta_A$ and $\mu_A\circ (\Id_H\otimes m_A)=m_A\circ \mu_{A\otimes A}$, where
  \begin{align*}
  \mu_{A\otimes A}=(\mu_A\otimes \mu_A)\circ (\Id_H\otimes c_{H,A}\otimes \Id_A)\circ (\Delta_H\otimes \Id_A\otimes \Id_A).
  \end{align*}
\end{itemize}

For an action $\mu_A$ and a morphism $\sigma\in \text{Reg}(H\otimes H, A)$, we formally define two morphisms $\eta_{A\#_{\sigma}H}$ and $m_{A\#_{\sigma}H}$ as follows
\begin{align*}
\eta_{A\#_{\sigma}H}:=\eta_A\otimes \eta_H,
\end{align*}
and
\begin{align*}
m_{A\#_{\sigma}H}:=& (m_A\otimes \Id_H)
\circ (m_A\otimes \sigma \otimes m_H)
\circ (\Id_A\otimes m_A \otimes \Delta_{H\otimes H})\\
& \circ (\Id_A\otimes \Id_H\otimes c_{H,A}\otimes \Id_H)
\circ (\Id_A\otimes \Delta_H\otimes \Id_A\otimes \Id_H).
\end{align*}
We say $(A\otimes H,m_{A\#_{\sigma}H},\eta_{A\#_{\sigma}H})$ is a \emph{crossed product} of $A$ and $H$ in $\mathcal{D}$, if it is an algebra in $\mathcal{D}$.

Now we recall $H$-cleft comodule algebra in a braided monoidal category $\mathcal{D}$.
\begin{definition}\cite[Definition 1.1]{AAj}
A right $H$-comodule algebra $(A, \rho_A)$ is called an $H$-cleft
comodule algebra if there exists an $H$-comodule morphism $\gamma$ in
\emph{Reg}$(H,A)$.
\end{definition}

Suppose $B=A\#_{\sigma}H$ a crossed product in $\mathcal{D}$.
Then $B$ is actually a $H$-cleft comodule algebra, where its comodule structure and the homomorphism $\gamma\in \text{Reg}(H,B)$ can be defined as follows:
\begin{align}\label{eq-cleft}
\rho_{B}:=\Id_A\otimes \Delta_H,\;\gamma:=\eta_A\otimes \Id_H.
\end{align}
Moreover, all $H$-cleft comodule algebras in a braided monoidal category with
equalizers and coequalizer is crossed products of an algebra $A$ and $H$.
\begin{theorem}\cite[Theorem 1.3]{AAj}\label{thm1.1}
Suppose $\mathcal{D}$ is a braided monoidal category with equalizers and coequalizer. Let $(A, \rho_A)$ be a right $H$-comodule algebra in $\mathcal{D}$.
Then the following statements are equivalent:
\begin{itemize}
  \item[(i)] $A$ is an $H$-cleft comodule algebra in $\mathcal{D}$;
  \item[(ii)] $A_0\#_{\sigma}H$ is a crossed product such that $A\cong A_0\#_{\sigma}H$ as $H$-comodule algebras in $\mathcal{D}$;
\end{itemize}
\end{theorem}
If we consider $\mathcal{D}=\mathbf{vect}_{\Bbbk}$, where $\mathbf{vect}_{\Bbbk}$ is just the braided monoidal category of vector spaces over $\Bbbk$, then all the results in this subsection return to the classical cases.

\section{Twisted Tensor Product}
In this section, we introduce the twisted tensor product of two quasitriangular Hopf algebras.

Recall that a normalized twist for a Hopf algebra $H$ is an invertible element $J\in H\otimes H$ which satisfying
\begin{align*}
(\epsilon \otimes \Id)(J)= (\Id \otimes \epsilon)(J)=1
\end{align*}
and
$$(\Delta \otimes \Id)(J)(J \otimes 1)=(\Id  \otimes \Delta)(J)(1 \otimes J).$$
With this twist, we can define a twisted Hopf algebra $H^J$ whose comultiplication $\Delta^J$ is given by
$$\Delta^J(h)=J\Delta(h)J^{-1},\quad h\in H$$
while its algebra structure remains the same as that of $H$.

According to \cite[Theorem 2.2]{NP-1}, two Hopf algebras $H$ and $H'$ are gauge equivalent (which means their representation categories are equivalent as tensor categories)
if and only if there exists a twist $J$ of $H$ such that $H'\cong H^J$ as Hopf algebras.
Furthermore, if both $(H,R)$ and $(H',R')$ are quasitriangular Hopf algebras, then their representation categories are equivalent as braided tensor categories if and only if there exists a twist $J$ of $H$ such that
$$(H',R')\cong (H^J,R^J)\text{ as quasitriangular Hopf algebras},$$
where $R^J$ is given by
$$R^J:=J_{21}RJ^{-1}.$$

\begin{definition}[Twisted Tensor Product]
Suppose $(H, R)$, $(K_1, R_1)$, $(K_2, R_2)$ are quasitriangular Hopf algebras.
We say $(H, R)$ is a twisted tensor product of $K_1$ and $K_2$ if there exists a twist $J$ of $K_1\otimes K_2$ such that
\begin{align*}
(H, R) \cong (K_1\otimes K_2)^{J}, \widetilde{R}^J)
\end{align*}
as quasitriangular Hopf algebras,
where $\displaystyle \widetilde{R}=\sum_{i,j}[(R_{1})_i\otimes (R_{2} )_{j}]\otimes [(R_{1})^i\otimes (R_{2})^{j}]$.
\end{definition}

If the Hopf algebras $K_1,K_2$ above are not isomorphic to $\Bbbk$, we say that $(H, R)$ is a \emph{non-trivial} twisted tensor product of them. Note that the Hopf algebra $(H, R)$ in this definition may be triangular, and we will consider some of these cases (eg. Theorem \ref{thm-m2}).

Suppose $(H, R)$ is a quasitriangular Hopf algebra and $(K, \pi)$ is a quotient Hopf algebra of $(H, R)$.
We will simply denote the corresponding braided monoidal category $(\Rep(H), \otimes, \Bbbk, c)$ as $\Rep(H)$, i.e. the representation category of $H$ with the braiding structure $\tau\circ R$, where $\tau$ is the flip map. Denote by $\Rep(K)$ the representation category of the quasitriangular Hopf algebra $(K,(\pi\otimes \pi)(R))$.
Through the surjective Hopf map $\pi$, we shall view $\Rep(K)$ as a natural braided monoidal subcategory of $\Rep(H)$.

Let
\begin{align}\label{eq-quo2}
L=\Phi_R(K^*)=\{((f\circ \pi) \otimes \id)(R_{21}\cdot R)|\;f\in K^*\}.
\end{align}
By Remark \ref{rk-nomco}, we see that $L$ is a normal left coideal subalgebra of $H$.
Therefore, $HL^+$ is a Hopf ideal and $H/HL^+(=K')$ is a quotient Hopf algebra, where
$$L^{+}=\{h\in L|\;\epsilon(h)=0\}=L\cap \ker \epsilon.$$
Let $\pi': H\to K'$ be the natural quotient map.
Then $(K',(\pi'\otimes \pi')(R))$ is also a quasitriangular Hopf algebra, leading to a corresponding braided monoidal category denoted by $\Rep(K')$.
Similarly, we view $\Rep(K')$ as a braided monoidal subcategory of $\Rep(H)$ through the quotient map $\pi'$.

Let $R_K:=(\pi\otimes \pi)(R)$ and $R_{K'}:=(\pi'\otimes \pi')(R)$.
Then, there is a natural quasitriangular structure on $K\otimes K'$ induced by $R_K,R_{K'}$ as follows:
\begin{equation}\label{eq-rdef}
R_{K\otimes K'}:=\sum_{i,j}[(R_K)_i\otimes (R_{K'})_j]\otimes [(R_K)^i\otimes (R_{K'})^j],
\end{equation}
where $\displaystyle R_K=\sum_i(R_K)_i\otimes (R_K)^i$ and $\displaystyle R_{K'}=\sum_j(R_{K'})_j\otimes (R_{K'})^j$.

A finite tensor category $(\mathcal{C}, \otimes, \mathds{1})$ \cite{EP} is defined as a rigid monoidal category where $\mathcal{C}$ is a finite abelian category.
Here ``finite" means that $\mathcal{C}(X,Y)$ are finite dimensional and $X$ has finite length for all $X,Y\in \text{Ob}(\mathcal{C})$, and every simple object of $\mathcal{C}$ has a projective cover, the number of isomorphism classes of simple objects of $\mathcal{C}$ is finite. The tensor product of $\mathcal{C}$ is $\Bbbk$-linear in each variable,
and the unit object of $\mathcal{C}$ is a simple object.
If $(\mathcal{C}, \otimes, \mathds{1})$ is equipped with a braiding $c$,
then it is termed a braided finite tensor category \cite{EP}. Denote $\Rep_f (H)$ as the braided tensor category consisting of finite dimensional representation of $H$ and the braiding structure is $\tau\circ R$.

Recall that a full subcategory of an abelian category is called \emph{topologizing subcategory} \cite{RA, Shi} if it is closed under finite direct sums and subquotients. By a \emph{braided tensor subcategory} of a braided tensor category $(\mathcal{C}, \otimes, \mathds{1}, c)$ we mean a subcategory of $\mathcal{C}$ containing the unit object of $\mathcal{C}$, closed under the tensor product of $\mathcal{C}$, and containing the braiding isomorphisms.

We shall recall that a finite braided tensor category $\mathcal{D}$ can factor through the relative M\"{u}ger centralizer as follows.
\begin{theorem}\cite[Theorem 4.17]{LC}\label{thm-category}
Let $(\mathcal{D}, \otimes, \mathds{1}, c)$ be a finite braided tensor category, let $\mathcal{E}$ be a topologizing non-degenerate braided tensor subcategory of $\mathcal{D}$. Then there is an equivalence of braided tensor categories:
$$\mathcal{D}\simeq \mathcal{E} \boxtimes C_{\mathcal{D}}(\mathcal{E}),$$
where
\begin{align*}
   \Ob(C_\mathcal{D}(\mathcal{E})) := \{Y \in \mathcal{D} | \;c_{Y,X} c_{X,Y}=\Id_{X\otimes Y} \text{ for all } X\in \mathcal{E}\}
\end{align*}
is the relative M\"{u}ger centralizer.
Moreover, if $\mathcal{D}$ is a ribbon category, then the above equivalence can be an equivalence of ribbon categories.
\end{theorem}
Note that semisimple case of this theorem is obtained by many authors (\cite[Theorem 8.21.4]{EP}).
Then, we have
\begin{lemma}\label{lem-mugcen}
 With the notations above and assume that $H$ is finite dimensional, the braided tensor subcategory $\Rep_f(K')$ is the M\"{u}ger centralizer of $\Rep_f(K)$ in $\Rep_f(H)$.
\end{lemma}

\begin{proof}
Let $X\in \text{Ob}(\Rep_f(K)), Y\in \text{Ob}(\Rep_f(H))$ such that $c_{Y,X} c_{X,Y}=\Id_{X\otimes Y}$.
Then we have $(R_{21}\cdot R)|_{X\otimes Y}=\Id_{X\otimes Y}$.
It's not difficult to see that $c_{Y,X} c_{X,Y}=\Id_{X\otimes Y}$ for all $X\in \text{Ob}(\Rep_f(K))$ if and only if $(R_{21}\cdot R)|_{K\otimes Y}=\Id_{K\otimes Y}$, where $K$ is the left regular represenation in $\text{Ob}(\Rep_f(K))$.
And the last equation is equivalent to $[(f\circ \pi \otimes \Id )(R_{21}\cdot R)]|_Y=\epsilon(f)\Id_Y$ for all $f\in K^*$.
Next, we only need to show that $(f\circ \pi \otimes \Id )(R_{21}\cdot R)]|_Y=\epsilon(f)\Id_Y$ if and only if $Y\in \text{Ob}(\Rep_f(K'))$.

Recall that $L=\Phi_R(K^*)$.
Then the equation $[(f\circ \pi \otimes \Id )(R_{21}\cdot R)]|_Y=\epsilon(f)\Id$ for all $f\in K^*$ is equivalent to $(HL^+).Y=0$, i.e. $Y\in \text{Ob}(\Rep_f(K'))$.
Thus, we achieve the desired conclusion.
\end{proof}
\begin{remark}
  The semisimple case of Lemma \ref{lem-mugcen} is already established in \cite{BSR}.
\end{remark}

\begin{theorem}\label{thm-ker}
Suppose $(H,R)$ is a finite dimensional quasitriangular Hopf algebra and $(K, \pi)$ is a quotient Hopf algebra of $(H, R)$.
Suppose that $(K, (\pi\otimes \pi)(R))$ is factorizable.
Let $K'=H/HL^+$ be the quotient Hopf algebra with $L$ defined in Equation \eqref{eq-quo2}.
Then there is a twist $J\in K\otimes K'$ such that
$$(H,R)\cong ((K\otimes K')^J,R_{K\otimes K'}^J) \text{ as quasitriangular Hopf algebras},$$
i.e. $(H, R)$ is a twisted tensor product of $(K, R_K)$ and $(K', R_{K'})$.
Furthermore, if $(H,R)$ is a ribbon Hopf algebra, then the isomorphism can also be an isomorphism of ribbon Hopf algebras.
\end{theorem}

\begin{proof}
Let $\mathcal{D}=\Rep_f(H)$ and let $\mathcal{E}=\Rep_f(K)$.
By definition, $\mathcal{E}$ is a topologizing braided tensor subcategory.
Since $(K,R_K)$ is factorizable, we know that $\mathcal{E}$ is non-degenerate. Now, we can apply Theorem \ref{thm-category} and Lemma \ref{lem-mugcen} to obtain our desired result.
\end{proof}

\section{Main Results}

In this section, we introduce a novel approach to studying the splitting property of quasitriangular Hopf algebras.
This method allows us to eliminate the restriction of finiteness, extending the applicability of the results.
We start by defining the concept of exact factorization in Hopf algebras.

\begin{definition}[Exact Factorization]
Suppose $H$ is a Hopf algebra. Let $L_1, L_2$ be left coideal subalgebras of $H$.
We say $H=L_1L_2$ is an exact factorization if the linear map $\varphi: L_1\otimes L_2\to H$ defined by $\varphi(u\otimes v)=uv$ is bijective, where $u\in L_1$, $v\in L_2$. Moreover, if $L_1,L_2$ normal, then we call this exact factorization is normal.
\end{definition}

\begin{theorem}\label{thm-bi}
Suppose that $(H, R)$ is a quasitriangular (not necessarily finite-dimensional) Hopf algebra.
The following statements are equivalent:
\begin{enumerate}[(1)]
\item There exist quotient Hopf algebras $(K_1, \pi_1)$, $(K_2, \pi_2)$ of $(H, R)$ such that $(H, R)$ is a twisted tensor product of $(K_1, \pi_1)$ and $(K_2, \pi_2)$ with twist $J$:
$$J=\sum_i (1\otimes \pi_2(S(R_i)))\otimes (\pi_1(R^i)\otimes 1).$$
\item There exist normal left coideal subalgebras $L_1, L_2$ of $(H, R)$ such that
    \begin{itemize}
\item $H=L_1L_2$ is an exact factorization;

\item $(\pi_1\otimes \pi_2)(R_{21}R)=1\otimes 1$, where $\pi_1$(resp. $\pi_2$) denote the associated natural quotient map with $H/HL_1^+$(resp. $H/HL_2^+$);
\end{itemize}

\end{enumerate}
\end{theorem}
\begin{proof}
(1)$\Rightarrow$(2):
We assume that $H=((K_1\otimes K_2)^J,R_{K_1\otimes K_2}^J)$.
Let $L_1=K_1\otimes 1\subseteq H$ and $L_2=S(1\otimes K_2)\subseteq H$.
It is clear that $K_1\otimes 1$ is normal left coideal subalgebra and  $1\otimes K_2$ is normal right coideal subalgebra of $(H, R)$, i.e. a right coideal subalgebra such that $S(h_1)kh_2\in K_2$ for all $h\in H,k\in K_2$.
This implies that $L_1,L_2$ are normal left coideal subalgebras of $(H, R)$. Moreover, $(\pi_1\otimes \pi_2)(R_{21}R)=1\otimes 1$ in this case. Now we see that the map $\varphi$ is bijective.
Thus, the normal left coideal subalgebras $L_1,L_2$ are obtained as required.

(2)$\Rightarrow$(1):
Let $K_1=H/HL_1^+$ and $K_2=H/HL_2^+$ and denote the associated natural quotient map by $\pi_1$ and $\pi_2$.
Let
\begin{align*}
J:=\sum_i (1\otimes \pi_2(S(R_i)))\otimes (\pi_1(R^i)\otimes 1)\in K_1\otimes K_2.
\end{align*}
By \cite[Lemma 4.2]{SchH}, we see that $J$ is a twist for the quasitriangular Hopf algebra $K_1\otimes K_2$.
We define $F: H\to K_1\otimes K_2$ as $F=(\pi_1\otimes \pi_2)\circ\Delta$.
Then $F$ is a Hopf map.
Moreover, $(F\otimes F)(R)=R_{K_1\otimes K_2}^J$ by using $(\pi_1\otimes \pi_2)(R_{21}R)=1\otimes 1$, where
\begin{align*}
R_{K_1\otimes K_2}^J=J_{21}R_{K_1\otimes K_2}J^{-1}.
\end{align*}

Next, we show $F$ is surjective.
Since $L_1$ is left coideal and $\pi_1(a)=\epsilon(a)1$ for $a\in L_1$, we have that $a\otimes 1= a_{(1)}\otimes \pi_1(a_{(2)})$ for all $a\in L_1$.
This implies for all $a\in L_1$,
\begin{align*}
1\otimes S^{-1}(a) =\pi_1 \big(S^{-1}(a)_{(1)}\big)\otimes S^{-1}(a)_{(2)}.
\end{align*}
Applying the Hopf map $F$, we see that for $a\in L_1$,
\begin{align*}
F(S^{-1}(a))=1\otimes S_{K_2}^{-1}\circ\pi_2(a).
\end{align*}
Therefore $S_{(K_1\otimes K_2)^J}(1\otimes \pi_2(L_1))\subseteq \im(F)$.
By the fact that  $\varphi$ is bijective, we have that $\pi_2|_{L_1}$, $\pi_1|_{L_2}$ are bijective.
Thus, $S_{(K_1\otimes K_2)^J}(1\otimes K_2)\in \im(F)$.
Similarly we have $F(b)=\pi_1(b)\otimes 1$ for $b\in L_2$.
Thus, we obtain that $K_1\otimes 1\in \im(F)$.
By the fact that $F$ is Hopf map and  $(K_1\otimes K_2)^J=(K_1\otimes 1)S_{(K_1\otimes K_2)^J}(1\otimes \pi_2(L_1))$,  we see that $F$ is surjective.

Now we show that $F$ is injective.
Let $\{x_i|i\in I\}$ be a linear basis of $L_1$ and $\{y_j|j\in I'\}$ a linear basis of $L_2$.
Note that $\pi_2|_{L_1}, \pi_1|_{L_2}$ are bijective.
We see that $\{S_{K_2}^{-1}\circ \pi_2(x_i)|i\in I\}$ is a linear basis of $1\otimes K_2$ and $\{\pi_1(y_j)|j\in I'\}$ is a linear basis of $K_1\otimes 1$.
Let $h\in H$ such that $F(h)=0$.
Since $L_1$ is normal, we obtain $H=L_2L_1$.
Thus, we write $\displaystyle h=\sum_{i,j =1}^n\lambda_{ij}y_jx_i$ for some $n\in \mathbb{N}^+$ and $\lambda_{ij}\in \Bbbk$.
Note that $F(a)=S_{(K_1\otimes K_2)^J}(1\otimes S_{K_2}^{-1}\circ\pi_2(a))$ for $a\in L_1$ and $F(b)=\pi_1(b)\otimes 1$ for $b\in L_2$, we obtain
$$F\left(\sum_{i, j=1}^n\lambda_{ij}y_jx_i\right)=\sum_{i, j=1}^n \lambda_{ij}[\pi_1(y_j)\otimes 1][S_{(K_1\otimes K_2)^J}(1\otimes S_{K_2}^{-1}\circ\pi_2(x_i))].$$
$L_1'=K_1\otimes 1\subseteq (K_1\otimes K_2)^J$ and let $L_2'=S_{(K_1\otimes K_2)^J}(1\otimes K_2)\subseteq (K_1\otimes K_2)^J$.
By statement (1), we have that the following map is bijective
\begin{equation*}
      \begin{matrix}
        \varphi':& L_1'\otimes L_2' &\rightarrow & (K_1\otimes K_2)^J \\
         \;&u\otimes v & \mapsto & uv
      \end{matrix}
       \end{equation*}
Combing this with the fact that $\{1\otimes S_{K_2}^{-1}\circ \pi_2(x_i)|i\in I\}$ is a linear basis of $1\otimes K_2$ and $\{\pi_1(y_j)\otimes 1|j\in I'\}$ is a linear basis of $K_1\otimes 1$, we get that $\lambda_{ij}=0$ for all $i,j=1, \ldots, n$.
This implies $h=0$, and hence $F$ is injective.
\end{proof}

\begin{remark}
It is worth to point out that Theorem \ref{thm-bi} is true for infinite-dimensional quasitriangular Hopf algebra $(H, R)$.
\end{remark}

\begin{proposition}\label{lem-bra}
Suppose $\mathcal{D}$ is a braided tensor category with equalizers and coequalizer and $H, K$ Hopf algebras in $\mathcal{D}$ with a surjective Hopf map $\pi: H\to K$.
If there exist a Hopf algebra $H_1$ and Hopf maps $i: K\to H_1$, $F_H: H_1\to H$ such that $\pi\circ F_H=i$, i.e. the following diagram commutes:
 \begin{equation}
\xymatrix{
    H\ar[r]^{\pi} & K \ar[ld]^{i} \\
    H_1\ar[u]^{F_{H}}  &
     }
\end{equation}
Then $(H,\rho)$ is a crossed product of $H^{co\pi}$ and $K$, where $\rho=(\Id\otimes \pi)\circ \Delta$.
\end{proposition}
\begin{proof}
We define the map $\gamma:K\rightarrow H$ by $\gamma= F_H\circ i$. Since the coaction of $K$ on $H$ is given by $(\id\otimes \pi) \circ \Delta$,  we know that $\gamma$ is right $K$-comodule map.
By a direct computation, we have that $\gamma *  (S\circ\gamma)=\eta_H \circ \epsilon_K$ and $(S\circ\gamma)  *  \gamma=\eta_H \circ \epsilon_K$, i.e. $\gamma^{-1}=S\circ \gamma$ in $\emph{\text{Reg}}(K, H)$.
Thus, $H^{co\pi}\hookrightarrow H$ is cleft extension in $\mathcal{D}$.
Now, applying Theorem \ref{thm1.1}, we see that $H$ is a crossed product of $H^{co\pi}$ and $K$.
\end{proof}



\begin{lemma}\label{lem-bijective}
Suppose $\mathcal{D}=\Rep(H)$ for some quasitriangular Hopf algebra $(H,R)$. Let $C$ be a Hopf algebra
in $\mathcal{D}$. Let $(A, \rho_A)$ be an $C$-cleft comodule algebra in $\mathcal{D}$ with a $C$-comodule morphism $\gamma$ in \emph{Reg}$(C,A)$. Then the map $\varphi: \rho^{\text{coinv}} \otimes C \to A$ defined by $\varphi=m_A \circ (i_{\rho^{\text{coinv}}}\otimes \gamma)$ is bijective linear map,
where $\rho^{\text{coinv}}=\ker(\rho_A-\Id_A\otimes \eta_C)$ and $i_{\rho^{\text{coinv}}}$ is the natural inclusion from $\rho^{\text{coinv}}$ into $A$.
\end{lemma}

\begin{proof}
Since $(A, \rho_A)$ be an $C$-cleft comodule algebra in $\mathcal{D}$, $A\cong \rho^{\text{coinv}} \#_{\sigma}C$ as $C$-comodule algebras in $\mathcal{D}$ by Theorem \ref{thm1.1}. Thus, we only need to consider the standard case, which is $A=\rho^{\text{coinv}} \#_{\sigma}C$ as $C$-comodule algebras with the map $\gamma$ defined by $\gamma=\eta_{\rho^{\text{coinv}}} \otimes \Id$. Direclty, the map $\varphi=m \circ (i_{\rho^{\text{coinv}}}\otimes \gamma)$ is just the identity map, which implies that $\varphi$ is bijective.
\end{proof}

\begin{theorem}\label{thm-m1}
Suppose that $(H, R)$ is a quasitriangular Hopf algebra (may infinite dimension) and $(K,\pi)$ a quotient Hopf algebra of $(H, R)$.
Suppose that $(K,(\pi\otimes \pi)(R))$ is factorizable.
Then there exists a quotient Hopf algebra $(K', \pi')$ of $(H, R)$ such that $(H, R)$ is a twisted tensor product of $(K, \pi)$ and $(K', \pi')$.
\end{theorem}
\begin{proof}
Let $\mathcal{D}=\Rep(H)$ be the braided monoidal category.
We define $\tilde{i}: \underline{K^*}\to \underline{H^o}$ as
\begin{align*}
    \tilde{i}(f)= f\circ \underline{\pi}, \quad f \in \underline{ K^*}.
\end{align*}
Then the following diagram in $\mathcal{D}$
\begin{equation}
\xymatrix{
    \underline{H}\ar[r]^{\underline{\pi}} & \underline{K} \\
    \underline{H^o}\ar[u]^{\underline{\Phi_R}}& \underline{K^* }\ar[u]^{\underline{\Phi_{\overline{R}}}} \ar[l]_{\tilde{i}}
     }
\end{equation}
commutes.
It is clear that $\underline{\Phi_R}, \underline{\Phi_{\overline{R}}}, \tilde{i}$ are Hopf maps in $\mathcal{D}$.
By Proposition \ref{prop-hopf}, we see that $\underline{\pi}$ is a Hopf map in $\mathcal{D}$.
By the assumption that $(K,\overline{R})$ is a factorizable Hopf algebra, we see that $\underline{\Phi_{\overline{R}}}$ is an isomorphism.
We also have $\im(\underline{\Phi_R} \circ \tilde{i})$ is isomorphic to $\underline{K}$ as Hopf algebra in $\mathcal{D}$.
By Proposition \ref{lem-bra},  we have $\underline{H}\cong \underline{H^{co\pi}}\#_{\sigma}\underline{K}$ as $\underline{K}$-comodule algebras in $\mathcal{D}$ for some $\sigma\in \emph{\text{Reg}}(\underline{K}\otimes \underline{K},  \underline{H^{co\pi}})$.
Moreover, due to the proof of Proposition \ref{lem-bra}, the map $\gamma=\underline{\Phi_R} \circ \tilde{i}\circ \underline{\Phi_{\overline{R}}}^{-1}$ is right $\underline{K}$-comodule map in $Reg(\underline{K},\underline{H})$. Now, we can use Lemma \ref{lem-bijective} to obtain that the following map is bijective:
 \begin{equation*}
      \begin{matrix}
        \varphi':& \underline{H^{co\pi}}\otimes \im(\gamma) &\rightarrow & \underline{H} \\
         \;&u\otimes v & \mapsto & uv
      \end{matrix}.
       \end{equation*}

By Theorem \ref{thm-bi}, we only need to show that there exist normal left coideal subalgebras $L_1,L_2$ such that the following map is bijective and they satisfy $(\pi_1\otimes \pi_2)(R_{21}R)=(1\otimes 1)$, where $\pi_1$(resp. $\pi_2$) denote the associated natural quotient map with $H/HL_1^+$(resp. $H/HL_2^+$)
      \begin{equation*}
      \begin{matrix}
        \varphi:& L_1\otimes L_2 &\rightarrow & H \\
         \;&u\otimes v & \mapsto & uv
      \end{matrix}
       \end{equation*}
Let $L_1=H^{co\pi}$ and let $L_2=\Phi_R(K^*)$ (See Equation \eqref{eq-quo2} for the definition).
We obtain that $L_1,L_2$ are normal left coideal subalgebras of $(H, R)$. To show $\varphi$ is bijective, we only need to show that $\varphi=\varphi'$ as linear maps via identifying $\underline{H}$ and $H$ as vector spaces.
By Lemma \ref{lem-copp}, we get $\underline{H^{co\pi}}=H^{co\pi}$ as vector spaces. Since the definition of $\gamma$, we know $\im(\gamma)=\Phi_R(K^*)$ as vector spaces. Thus, the map $\varphi$ is just $\varphi'$ as linear maps, which implies that $\varphi$ is bijective. Since $L_2=\Phi_R(K^*)$ and the definition of $\pi_1$, we have $\pi_1(\Phi_R(f))=\epsilon(f)1$ for any $f\in K^*$, this implies $(\pi_1\otimes \pi_2)(R_{21}R)=(1\otimes 1)$. Hence by Theorem \ref{thm-bi}, $(H, R)$ is a twisted tensor product of $(K, \pi)$ and $(K', \pi')$ for some quotient Hopf algebra $(K', \pi')$ of $(H, R)$.
\end{proof}

Recall the definition of Drinfeld double of a finite dimensional Hopf algebra $H$ denoted by $D(H)$ (see \cite{MS}). As a coalgebra, $D(H)=(H^*)^{cop}\otimes H$. The multiplication of $D(H)$ is defined as
$$(f\otimes h)(g\otimes k)=f [h_{(1)}\rightharpoonup g\leftharpoonup S^{-1}(h_{(3)})]\otimes h_{(2)}k,$$
where $f,g\in H^*, \;h,k\in H$ and $\langle a\rightharpoonup g\leftharpoonup b,c\rangle:=\langle g,bca\rangle$ for $a,b,c\in H$. Then, we have
\begin{corollary}\cite[Theorem 4.3]{SchH}
Suppose $H=D(K)$ for a factorizable Hopf algebra $(K,R)$.
Then $D(K)\cong (K\otimes K)^J$ as quasitriangular Hopf algebras, where the twist $J$ is given by
$$J=\sum_i(1\otimes R^i)\otimes (R_i\otimes 1).$$
\end{corollary}

\begin{proof}
Denote the standard universal $R$-matix of $D(K)$ as $\mathcal{R}$. Let $\pi:H\rightarrow K$ be the surjective Hopf map given by
\begin{align*}
\pi(f\otimes k)=S(r_R(f))k, \quad  f\in K^*, k\in K.
\end{align*}
Then $(K,(\pi\otimes \pi)(\mathcal{R}))$ is factorizable. Thus, from Theorem \ref{thm-m1} and Theorem \ref{thm-bi}, we know this corollary is true.
\end{proof}

We propose a conjecture of the categorical version of Theorem \ref{thm-m1} in the following.
\begin{conjecture}
Suppose $\mathcal{D}$ is a braided tensor category and $\mathcal{C}$ is a topologizing non-degenerated braided tensor subcategory of $\mathcal{D}$. Then
\begin{align*}
    \mathcal{D}\cong \mathcal{C}\boxtimes C_{\mathcal{D}}(\mathcal{C}).
\end{align*}
as braided tensor category.
\end{conjecture}

\begin{remark}
\emph{Here, we don't ask that the braided tensor category $\mathcal{D}$ has left dual or right dual. Let $(H,R)$ be a quasitriangular Hopf algebra with a factorizable quotient Hopf algebra $(K,\pi)$. When $\mathcal{D}=\Rep(H)$ and $\mathcal{C}=\Rep(K)$, we see this conjecture holds by Theorem \ref{thm-m1}.} 
\end{remark}

For the full rank quasitriangular Hopf algebra, we also have splitting property for quasitriangular Hopf algebras.
\begin{theorem}\label{thm-m2}
Suppose $(H,R)$ is a quasitriangular Hopf algebra and $(K, \pi)$ is a quotient Hopf algebra of $(H, R)$.
Suppose $H^{co\pi}=\prescript{co\pi}{}{H}$ and $(K,(\pi\otimes \pi)(R))$ is full rank quasitriangular Hopf algebra.
Then there exists a quotient Hopf algebra $(K', \pi')$ of $(H, R)$ such that $(H, R)$ is a twisted tensor product of $(K, \pi)$ and $(K', \pi')$.
\end{theorem}

\begin{proof}
Similar to the proof of Theorem \ref{thm-m1}, we only need to show that there exist normal left coideal subalgebras $L_1,L_2$ such that the following map is bijective
      \begin{equation*}
      \begin{matrix}
        \varphi:& L_1\otimes L_2 &\rightarrow & H \\
         \;&u\otimes v & \mapsto & uv
      \end{matrix}
       \end{equation*}
Let $L_1=H^{co\pi}$ and let $L_2=l_R(K^*)$, where $l_R(K^*)=\{l_R(f\circ \pi)|\;f\in K^*\}$. We first show the map $\varphi$ is bijective. Let $\overline{R}=(\pi\otimes \pi)(R)$. Considering following diagram
\begin{equation}
\xymatrix{
    H\ar[r]^{\pi} & K \\
    (H^o)^{cop}\ar[u]^{l_R}& (K^*)^{cop} \ar[u]^{l_{\overline{R}}} \ar[l]_{i}
     }
\end{equation}
where $i(f)=f\circ \pi$ for $f\in K^*$. Since all the maps in this diagram are Hopf maps, we see that the above diagram is commutative.  By the assumption that $(K,\overline{R})$ is full rank quasitriangular Hopf algebra, we know that $l_{\overline{R}}$ is bijective. Thus, the conditions of Proposition \ref{lem-bra} hold. Thus, $H\cong H^{co\pi}\#_{\sigma} K$ as $K$-comodule algebras for some $\sigma\in \text{Reg}(K\otimes K, H^{co\pi})$. In this case, the map $\gamma$ in $Reg(K,H)$ can be given by $l_R\circ i\circ l_{\overline{R}}^{-1}$ by the proof of Proposition \ref{lem-bra}. Thus, $\im(\gamma)=l_R(K^*)$. Now, we can apply Lemma \ref{lem-bijective} to deduce that the map $\varphi$ is bijective.

Note that $L_1$ is left normal coideal subalgebra.
We only have to show $L_2$ is also a left normal coideal subalgebra and $(\pi_1\otimes \pi_2)(R_{21}R)=(1\otimes 1)$, where $\pi_1$(resp. $\pi_2$) denote the associated natural quotient map with $H/HL_1^+$(resp. $H/HL_2^+$).
Recall that $L_2=l_R(K^*)$ is Hopf subalgebra of $H$, we have to prove that $L_2$ is normal.
Since $\varphi$ is bijective, we get $H=L_1L_2$.
Next, we show $(ab)_{(1)}cS((ab)_{(2)})\in L_2$ for all $a\in L_1$ and $b,c\in L_2$.
Since $L_2$ is Hopf subalgebra, we obtain that $b_{(1)}cS(b_{(2)})\in L_2$.
Thus, we can assume $b_{(1)}cS(b_{(2)})=l_R(f\circ \pi)$ for some $f\in K^*$.
Since
\begin{align*}
(g_{(1)}\rightharpoonup h)l_R(g_{(2)})=l_R(g_{(1)})(h\leftharpoonup g_{(2)})
\end{align*}
for all $g\in H^o, h\in H$, we know $hl_R(g)=l_R(g_{(2)})(g_{(1)}\circ S^{-1}\rightharpoonup h\leftharpoonup g_{(3)})$ for all $g\in H^o, h\in H$.
Using this, we have
\begin{align*}
a_{(1)}l_R(f\circ \pi)S(a_{(2)})&=[a_{(1)}l_R(f\circ \pi)]S(a_{(2)})\\
                &=[l_R(f_{(2)}\circ \pi) (f_{(1)}\circ \pi\circ S^{-1}\rightharpoonup a_{(1)}\leftharpoonup f_{(3)}\circ \pi)]S(a_{(2)})\\
                &=\langle f_{(1)}\circ S_K^{-1}, \pi(a_{(3)})\rangle \langle f_{(3)}, \pi(a_{(1)})\rangle l_R(f_{(2)}\circ \pi) a_{(2)}S( a_{(4)})\\
                &=\langle f_{(1)}\circ S_K^{-1}, 1_K)\rangle \langle f_{(3)}, 1_K)\rangle l_R(f_{(2)}\circ \pi) a_{(1)}S( a_{(2)}),
\end{align*}
where the last equation follows from $H^{co\pi}=\prescript{co\pi}{}{H}$.
Then we see that
\begin{align}
a_{(1)}l_R(f\circ \pi)S(a_{(2)})=l_R(f\circ \pi)\epsilon(a). \label{commu}
\end{align}
Thus, we have shown that $L_2$ is normal. Further, one can see that $H^{co\pi}$  commutes with $l_R(K^*)$ due to above equality \eqref{commu}. Thus, $H=H^{co\pi}.l_R(K^*)$ has actually tensor product structure as Hopf algebras. From this, we can assume that $H=A\otimes B$ as Hopf algebras and
$$\pi_1=(\Id_A \otimes \epsilon_B), \;\pi_2=(\epsilon_A\otimes \Id_B),$$
where $A=H^{co\pi}$ and $B=l_R(K^*)$. From this, we can assume $R=\sum_{i,j}(a_i\otimes b_j)\otimes (a^i\otimes b^j)$, where $a_i,a^i\in A$ and $b_i,b^i\in B$. Then, we have:
\begin{align*}
(\pi_1\otimes \pi_2)(R_{21}R)&=(\pi_1\otimes \pi_2)[\sum_{k,l}(a^k\otimes b^l)\otimes (a_k\otimes b_l)][\sum_{i,j}(a_i\otimes b_j)\otimes (a^i\otimes b^j)]\\
&=\epsilon_A(a_k)\epsilon_B(b_j)\epsilon_A(a^i)\epsilon_B(b^l)(a^k a_i\otimes b_lb^j)\\
&=\epsilon_A(a^i)\epsilon_B(b^l)(a_i\otimes b_l)\\
&=1\otimes 1,
\end{align*}
where the last equality follows from $\sum_{i,l}(a_i\otimes b_l)\otimes (\epsilon_A(a^i)\otimes \epsilon_B(b^l))=1\otimes 1$ and the penultimate equality follows from an analogous equation. Thus, we have $(\pi_1\otimes \pi_2)(R_{21}R)=(1\otimes 1)$.

Note that $H^{co\pi }\cong H/HL_2^+$.
Combining this with Theorem \ref{thm-bi}, we obtain the desired results.
\end{proof}

In the end of this section, we shall construct an example to illustrate that there exists the case that Theorem \ref{thm-m1} can't work while Theorem \ref{thm-m2} can apply.

Suppose the characteristic of the ground field $\Bbbk$ is not $2$.
We recall that the Sweedler Hopf algebra $H_4$ .
As an algebra, $H_4$ is generated by $a$ and $x$ subject to the relations
\begin{align*}
a^2 = 1, \quad x^2 = 0, \quad xa = -ax.
\end{align*}
The coalgebra structure of $H_4$ is determined by
\begin{align*}
\Delta(a)=a\otimes a, \quad \Delta(x)=x\otimes a+ 1\otimes x.
\end{align*}
By the fact that all the universal $R$-matrices of $H_4$ are determined, see \cite{Ge-Taft} or \cite[Section 12.2]{R} for example. we know that $H_4$ is not factorzable, i.e. there is no univeral $R$-matix of $H_4$ such that it gives factorizable Hopf algebra. Let $H=H_4\otimes \Bbbk \mathbb{Z}_2$. By dimension argument for $H$, we know that there is no non-trivial factoizable quotient of $H$, i.e. Theorem \ref{thm-m1} can't apply to know if $H$ has any non-trivial twisted tensor product decomposition. Since all the universal $R$-matrices of $H_4$ are determined, we obtain that there exists a univeral $R$-matix $R$ of $H_4$ such that $(H_4,R)$ is full rank. Hence the natural quotient map $\pi:H\rightarrow H_4$ given by $\pi(b\otimes c)=b \epsilon(c)$ satifies the conditions of Theorem \ref{thm-m2}. Thus, we can apply Theorem \ref{thm-m2} to obtain twisted tensor products of $H_4$ and $\mathbb{Z}_2$.

\section{Extensions of Hopf Algebras}

We shall recall the definition of the extension of Hopf algebras,
which have been studied by many authors (see \cite{M3} for example).
\begin{definition}[Extension of Hopf Algebras]\label{def2.1.1}
Suppose
\begin{equation}\label{ext}
\;\; A\xrightarrow{\iota} H \xrightarrow{\pi} K
\end{equation}
is a sequence of Hopf algebras and the maps are Hopf  maps.
We say $H$ is an extension of $A$ by $K$ if
\begin{itemize}
  \item[(i)] $\iota$ is injective;
  \item[(ii)]  $\pi$ is surjective;
 \item[(iii)] $A=\{h \in H|\;(\pi\otimes \Id)\circ \Delta (h)=1\otimes h\}=H^{co \pi}$;
  \item[(iv)] $\ker(\pi)= HA^+$, $A^+$ is the kernel of the counit of $A$.
\end{itemize}
where $A$ is viewed as a Hopf subalgebra of $H$ via the map $\iota$.
An extension $H$ of $A$ by $K$ is abelian if $A$ is commutative and $K$ is cocommutative.
\end{definition}

For an exact sequence \eqref{ext}, we say that $A$ is a normal Hopf
subalgebra of $H$, i.e. $h_{(1)}kS(h_{(2)})\in K$ for all $k\in K,h\in H$.


When the extension $H$ is abelian, $H$ is finite dimensional and $\Bbbk$ is algebraically closed of characteristic zero, we see that there exist finite groups $G$, $F$ such that
\begin{equation}\label{eq-abel}
\;\; \Bbbk^G\xrightarrow{\iota} H \xrightarrow{\pi} \Bbbk F,
\end{equation}
Note that the above abelian extensions were classified by Masuoka
(see \cite[Proposition 1.5]{M3}).

\begin{theorem}\label{pro-ext}
Suppose $(H,R)$ is a quasitriangular Hopf algebra satisfying the following extension:
\begin{equation}\label{apply-ext}
A\xrightarrow{\iota} H \xrightarrow{\pi} K,
\end{equation}
If $(K,(\pi\otimes \pi)(R))$ is factorizable or full rank, then the Hopf algebra $A$ admits a quasitriangular structure $R_A$ and there exists a twist $J$ for $K\otimes A$ such that $(H, R)$ is a twisted tensor product of $(K, (\pi\otimes \pi)(R))$ and $(A, R_A)$.
\end{theorem}
\begin{proof}
Assume $(K,(\pi\otimes \pi)(R))$ is factorizable. Note that $A=H^{co\pi}$ is a Hopf subalgebra in this case, hence we can use the proof of Theorem \ref{thm-m1} to obtain what we want.

Similarly, if $(K,(\pi\otimes \pi)(R))$ is full rank, then we can use the fact that $H^{co\pi}=\prescript{co\pi}{}{H}=A$ to obtain that the quotient map $\pi$ satisfies the conditions of Theorem \ref{thm-m2}, now we can use the proof of Theorem \ref{thm-m2} to get we want.
\end{proof}

\begin{proposition}
Suppose $(H,R)$ is a quasitriangular Hopf algebra and an abelian extension of a Hopf algebra $A$ by a Hopf algebra $K$.
Suppose that $\Bbbk$ is algebraically closed of characteristic
$0$ and $(K,(\pi\otimes \pi)(R))$ is factorizable or full rank.
Then, there exists an abelian group $G$ and abelian Lie algebra $L$ such that
\begin{align*}
   H\cong U(L)\otimes \Bbbk G
\end{align*}
as Hopf algebras, where $U(L)$ is universal enveloping algebra of $L$.
Furthermore, if $H$ is finite dimensional, then $H$ is a finite abelian group algebra and the extension is a direct product of groups.
\end{proposition}

\begin{proof}
By Theorem \ref{pro-ext}, $A$ has quasitiriangular structure. Since $A$ is also commutative, $A$ is cocommutative.
Since $\Bbbk$ is algebraically closed of characteristic $0$, we can apply the Cartier-Kostant-Milnor-Moore theorem (see \cite[Section 5]{MS}) to obtain $A\cong U(L)\# \Bbbk G$ for some Lie algebra $L$ and some group $G_0$.
Since $A$ is commutative, we obtain that $L$ is abelian Lie algebra and $G_0$ is abelian group.
Moreover, we have $A\cong U(L)\otimes \Bbbk G_0$.
Since $K$ is factorizable or full rank, $K$ is finite dimensional. Combing this with other assumptions about $K$, we obtain that $K\cong \Bbbk F$ for some finite abelian group $F$.
By Theorem \ref{pro-ext} again, we know that $H\cong [U(L)\otimes \Bbbk (G_0\times F)]^J$ for some twist $J$. Due to commutativity of $U(L)\otimes \Bbbk (G_0\times F)$, we get $H\cong U(L)\otimes \Bbbk (G_0\times F)$. Let $G=G_0\times F$. Then we have $H\cong U(L)\otimes \Bbbk G$ as what we want. From this, we see that if $H$ is finite dimensional, then $H\cong \Bbbk G$ and the extension is a direct product of groups. Thus, we complete the proof.
\end{proof}

\section{Criterion for Quasitriangular Hopf Algebras}

In this section, we will provide a criterion for finite dimensional quasitriangular Hopf algebras.

 Let $f\in H^*$ and let $h\in H$, define the actions as follows:
$$f\rightharpoonup h:=h_{(1)}f(h_{(2)}),\;h\leftharpoonup f:=f(h_{(1)})h_{(2)}.$$

\begin{lemma}\cite[Proposition 12.2.11]{R}\label{lem-comm}
Suppose that $(H, R)$ is a finite dimensional quasitriangular Hopf algebra. Then, for all $h\in H$ and $f\in H^*$, we have
$$(f_{(1)}\rightharpoonup h)l_R(f_{(2)}) = l_R(f_{(1)})(h\leftharpoonup f_{(2)}).$$
\end{lemma}


\begin{proposition}\label{pro-prop}
Suppose $(H, R)$ is a quasitriangular Hopf algebra.
Then $H$ must satisfies one of following conditions:
\begin{itemize}
 \item[(i)] $|G(H^o)|=1$;
 \item[(ii)] $|G(H^o)|\neq 1$ and $G(H^o)\cap Z(H^o)\neq\{1\}$;
 \item[(iii)] $|G(H^o)|\neq 1$ and $G(H^o)\cap Z(H^o)= \{1\}$ and $Z(G(H))\neq\{1\}$ ;
\end{itemize}
 where $Z(H)$ (resp. $Z(H^o)$) is the center of $H$ (resp. $H^o$).
\end{proposition}

\begin{proof}
Suppose $|G(H^o)|\neq 1$. Let $1\neq \alpha\in G(H^o)$. If $l_R(\alpha)=1$, using Lemma \ref{lem-comm} by letting $f=\alpha$, we conclude that $\alpha\rightharpoonup h=h\leftharpoonup \alpha$ for all $h\in H$. From this, we know $\alpha\in G(H^o)\cap Z(H^o)$, i.e. (ii) holds. If $l_R(\alpha)\neq 1$, then $l_R(\alpha)\in G(H)$ and it belongs to the center of $G(H)$, which implies (iii).
\end{proof}

\begin{proposition}\label{pro-qt}
Suppose $(H,R)$ is a finite dimensional quasitriangular Hopf algebra. Then $H$ must satisfies one of following conditions:
\begin{itemize}
 \item[(i)] $|G(H^*)|=1$;
 \item[(ii)] $|G(H^*)|\neq 1$ and $G(H^*)\cap Z(H^*)\neq \{1\}$;
 \item[(iii)] $|G(H^*)|\neq 1$ and $G(H^*)\cap Z(H^*)=\{1\}$ and $G(H)\cap Z(H)\neq \{1\}$;
  \item[(iv)] $|G(H^*)|\neq 1$ and $G(H^*)\cap Z(H^*)=\{1\}$ and $G(H)\cap Z(H)= \{1\}$ and for any odd prime factor $p$ of $|G(H^*)|$, there exists $g\in G(H)$ with order $p$ and $\alpha \in G(H^*)$ with order $p$ such that $\alpha(g)= 1$; In particular, $p^2|\dim(H)$.
\end{itemize}
\end{proposition}

\begin{proof}

By Proposition \ref{pro-prop}, we only need to show that if $H$ don't satisfy (i)-(iii), then $H$ must satisfy (iv).
Suppose (i)-(iii) do not hold.
Let $p$ be any odd prime factor of $|G(H^*)|$.
Then we can find some $\alpha\in G(H^*)$ with order $p$.
Denote the Hopf subalgebra generated by $\alpha$ as $A$.
Then we have $A\cong \Bbbk \mathbb{Z}_p$ and there is a natural surjective Hopf map
\begin{align*}
\begin{matrix}
\pi_1:& H &\rightarrow & A^* \\
 \;&h & \mapsto & (f\mapsto f(h))
\end{matrix}
\end{align*}
Let $\langle \alpha, \alpha\rangle=(\alpha\otimes \alpha)(R)$.
If $\langle \alpha, \alpha\rangle\neq 1$, then $\langle \alpha, \alpha\rangle$ is primitive $p$th root of $1$.
Hence, $(A^*,(\pi_1\otimes \pi_1)(R))$ is a factorizable Hopf algebra.
Applying Theorem \ref{thm-ker}, we know that there exists a Hopf algebra $B$ and there is a twist $J$ for $(B\otimes A^*)$ such that
$$H \cong (B\otimes A^*)^J.$$
Note that $B\cong \Bbbk \mathbb{Z}_p$ is a commutative Hopf subalgebra of $(B\otimes A^*)^J$.
Thus, there is some $1\neq g\in G(H)\cap Z(H)$, which contradicts with the condition (iii).
Thus,  $\langle \alpha, \alpha\rangle=1$.
By the proof of Proposition \ref{pro-prop}, we know that $l_R(\alpha)\neq 1$.
Let $g=l_R(\alpha)$.
Then the condition (iv) holds.
In this case, it's not hard to see that $p^2|\dim(H)$.
\end{proof}

\begin{theorem}\label{pro-qtt}
Suppose $H$ is a finite dimensional Hopf algebra.
Suppose $|G(H^*)|\neq 1$, $G(H^*)\cap Z(H^*)=\{1\}$, $G(H)\cap Z(H)=\{1\}$, and there exists an odd prime factor $p$ of $|G(H^*)|$ such that for any $g\in G(H)$ with order $p$ and any $\alpha \in G(H^*)$ with order $p$, $\alpha(g)\neq 1$.
Then $H$ don't have any universal $R$-matrix.
\end{theorem}

\begin{proof}
By Proposition \ref{pro-qt}, we complete the proof of theorem.
\end{proof}

\begin{example}\label{pro-exsit1}
\emph{Let $p$ be a odd prime number. Considering the $p^2$ dimensional Taft algebra denoted by $H_{p^2}$, i.e. $H_{p^2}$ is generated by $\{a,x\}$ subject to the relations
$$a^p=1, \quad x^p = 0,\; \quad xa = \omega ax, \quad  \omega \text{ is a } p\text{th} \text{ primitive root of } 1.$$
The coalgebra structure of $H_{p^2}$ is determined by
$$\Delta(a)=a\otimes a, \quad \Delta(x)=x\otimes a+ 1\otimes x.$$
Let $\alpha\in G(H_{p^2}^*)$ be the element determined by $\alpha(a)=\omega$. Directly, we have $G(H_{p^2}^*)=\langle \alpha \rangle$ as groups. Combing this with $G(H_{p^2})=\langle a \rangle$, the conditions of Theorem \ref{pro-qtt} hold. Hence, $H_{p^2}$ admits no universal $R$-matrix, which is implied by \cite{Ge-Taft}.}
\end{example}

A Hopf algebra $H$ is \emph{simple} if the only normal Hopf subalgebras are $\Bbbk$ and itself.
If a quasitriangular Hopf algebra is simple, then we briefly call a simple quasitriangular Hopf algebra \cite{GA, Nar, Naq}.

\begin{corollary}\label{coro-simple}
Suppose $(H,R)$ is a simple quasitriangular Hopf algebra with finite dimension.
Then one of the following conditions hold:
\begin{itemize}
 \item[(i)] $|G(H^*)|=1$;
  \item[(ii)] $|G(H^*)|\neq 1$, and for any odd prime factor $p$ of $|G(H^*)|$, there exists $g\in G(H)$ with order $p$ and $\alpha \in G(H^*)$ with order $p$ such that $\alpha(g)= 1$; In particular, $p^2|\dim(H)$.
\end{itemize}
\end{corollary}

\begin{proof}
Since $H$ is simple Hopf algebra with finite dimension , we know that $G(H^*)\cap Z(H^*)=\{1\}$ and $G(H)\cap Z(H)= \{1\}$.
Now applying Proposition \ref{pro-qt}, we see the corollary is true.
\end{proof}

\begin{corollary}\label{rk-sim}
Suppose $H$ is a self-dual non-semisimple simple Hopf algebra with odd dimension.
If there exists an odd prime factor $p$ of $|G(H^*)|$ such that for any $g\in G(H)$ with order $p$ and any $\alpha \in G(H^*)$ with order $p$, $\alpha(g)\neq 1$, then $H$ admits no universal $R$-matrix.
\end{corollary}
\begin{proof}
Since $H$ is non-semisimple Hopf algebra with odd dimension, we know $\text{Tr}|_{S^2}$ is zero.
This implies the distinguished group-like element of $H^*$ is not $1$ by Radford's fourth power formula, i.e. $|G(H^*)|\neq 1$.
Thus, we can apply Corollary \ref{coro-simple} to complete the proof.
\end{proof}

In the end of the section, we will apply Corollary \ref{coro-simple} to obtain a short proof of a result of Natale.
For readers' convenience, we state her result as follows:

\begin{theorem}\cite[Theorem 1.2]{Nar}\label{thm-squ}
Let $(H, R)$ be a finite dimensional quasitriangular Hopf algebra over an algebraically closed field of characteristic zero.
Assume that $\dim(H)$ is odd and square-free.
Then $H$ is semisimple.
\end{theorem}


\begin{proof}[Proof of Theorem \ref{thm-squ}]
Let $n$ be the number of prime factors of $\dim(H)$.
We will proceed by induction on $n$.
For $n=1$, by \cite{ZY}, $H$ is a group algebra, and therefore $H$ is semisimple.

Now, assume $H$ is semisimple for $n\leq k$.
We need to prove that $H$ is also semisimple for $n=k+1$.
Let $H_l:=l_R(H^*)=\{l_R(f)|\;f\in H^*\}$.  If $\dim(H_l)\neq\dim(H)$.
Assume that $p$ is an odd prime factor of $\dim(H)$ such that $p\nmid \dim(H_l)$.
This implies that $p\nmid \dim(D(H_l))$.
Note that $H_R$ is a quotient Hopf algebra of $D(H_l)$.
Hence, $p\nmid \dim(H_R)$.
This implies that $H_R$ is a proper Hopf subalgebra of $H$, i.e. the number of prime factors of $\dim(H_R)$ is less than $k+1$.
By the inductive hypothesis, $H_R$ is semisimple. Combining this with the assumption that $(H,R)$ is quasitriangular with odd dimension, we obtain that $H$ is also semisimple.

If $\dim(H_l)=\dim(H)$, then $l_R$ is an isomorphism.
In this case, if $H$ is not simple Hopf algebra, we can apply induction to show that $H$ must be an extension of semisimple Hopf algebras, implying $H$ is semisimple (\cite{BM}).

The remaining case is when $H$ is a simple Hopf algebra. By Corollary \ref{coro-simple} and dimensional arguments, we know that $|G(H^*)|=1$. Since $H\cong (H^*)^{cop}$, we also have $|G(H)|=1$, which implies $S^4=\Id$. As $\dim(H)$ is odd, we know that $\text{Tr}(S^2)\neq 0$. Therefore, by \cite[Proposition 10.4.2]{R}, $H$ is also semisimple.
\end{proof}

\begin{remark}
In S. Natale's proof in \cite{Nar}, she firstly studied the conormality of the map $\pi:H\rightarrow H/HL^+$, where $\pi$ is the natural quotient map and $L=\Phi_R(H^*)$.
Under these conditions, along with the classification of triangular Hopf algebras from \cite{PGTC}, she proved Theorem \ref{thm-squ}.
However, our proof primarily relies on Corollary \ref{coro-simple}, which is a byproduct of Theorem \ref{thm-m1}.
\end{remark}


\end{document}